\documentclass[leqno,12pt]{amsart} 
\usepackage{amssymb,enumerate}
\usepackage{pgf,tikz}
\usepackage{mathrsfs}

\theoremstyle{plain} 
\newtheorem{theorem}{\indent\sc Theorem}[section]
\newtheorem{lemma}[theorem]{\indent\sc Lemma}
\newtheorem{corollary}[theorem]{\indent\sc Corollary}

\newtheorem{claim}[theorem]{\indent\sc Claim}

\theoremstyle{definition} 

\newtheorem{remark}[theorem]{\indent\sc Remark}

\def\Pi{{\mathbf{P}}}

\begin{document}

\title[Spanning $k-$ended tree]{Progress on sufficient conditions for a graph to have a spanning $k-$ended tree} 

\author[P.~H.~Ha]{Pham Hoang Ha}
\keywords{spanning tree, leaf, sum degree.}

\subjclass[2010]{ 
Primary 05C05, 05C70. Secondary 05C07, 05C69.
}

\address{
Department of Mathematics, Hanoi National University of Education, 136, XuanThuy str., Hanoi, Vietnam
}
\email{ha.ph@hnue.edu.vn}
\maketitle
\begin{abstract}
In 1998, Broersma and Tuinstra [J. Graph Theory \textbf{29} (1998), 227-237] proved that if $G$ is a connected graph satisfying $\sigma_2(G) \geq |G|-k+1$ then $G$ has a spanning $k-$ended tree. They also gave an example to show that the condition "$\sigma_2(G) \geq |G|-k+1$" is sharp. In this paper, we introduce a new progress for this result.
Let $K_{m,m+k}$ be a complete bipartite graph with bipartition $V(K_{m,m+k})=A\cup B, |A|=m, |B|=m+k.$ Denote by $H$ to be the graph obtained from $K_{m,m+k}$ by adding (or no adding) some edges with two end vertices in $A.$ We prove that if $G$ is a connected graph satisfying $\sigma_2(G) \geq |G|-k$ then $G$ has a spanning $k-$ended tree except for the case $G$ is isomorphic to a graph $H.$ As a corollary of our main result, a sufficient condition for a graph to have a few branch vertices is given.	
\end{abstract}
\section{Introduction}

In this paper, we only consider finite simple graphs. Let $G$ be a
graph with vertex set $V(G)$ and edge set $E(G)$. For any vertex
$v\in V(G)$, we use $N_G(v)$ and $d_G(v)$ (or $N(v)$ and $d(v)$ if
there is no ambiguity) to denote the set of neighbors of $v$ and the
degree of $v$ in $G$, respectively. For any $X\subseteq V(G)$, we
denote by $|X|$ the cardinality of $X$. We
use $G-X$ to denote the graph obtained from $G$ by deleting the
vertices in $X$ together with their incident edges. We define $G-uv$ to be the
graph obtained from $G$ by deleting the edge $uv\in E(G)$, and
$G+uv$ to be the graph obtained from $G$ by adding an edge $uv$
between two non-adjacent vertices $u$ and $v$ of $G$. We write $A:=B$ to rename $B$ as $A$.

A subset $X\subseteq V(G)$ is called an \emph{independent set} of
$G$ if no two vertices of $X$ are adjacent in $G$.  The maximum size
of an independent set in $G$ is denoted by $\alpha(G)$. For $n\geq
1$, we define $\sigma_2(G)=\min\{ d(u)+
d(v)$ for all non-adjacent vertices $u, v \in V(G)$\}. 

Let $T$ be a tree. A vertex of degree one is a \emph{leaf} of $T$
and a vertex of degree at least three is a \emph{branch vertex} of
$T$. Setting $L(T)$ the set of leaves of the tree $T.$ A tree having at most $k$ leaves is called a $k-$ended tree. A spanning tree of a graph $G$ is a tree $T$ with $V(T)=V(G).$ In particular, a Hamiltonian path is a spanning $2-$ended tree. We refer
to~\cite{Di05} for terminology and notation not defined here.

There are several well-known conditions (such as the independence
number conditions and the degree sum conditions) ensuring that a
graph $G$ contains a spanning tree with a bounded number of leaves
or branch vertices (see the survey paper~\cite{OY} and the
references cited therein for details).

Ore~\cite{Ore} obtained a
sufficient condition related to the degree sum for
 a connected graph to have a Hamiltonian path. 
\begin{theorem}[{\cite[Ore]{Ore}}]\label{Ore}
	Let $G$ be a connected graph. If $\sigma_2(G) \geq |G|-1$, then $G$ has
	a Hamiltonian path.
\end{theorem}
\noindent After that, Broersma and Tuinstra~\cite{BT} generalized above result by proving the following theorem.
\begin{theorem}\label{t2}{\rm (Broerma and Tuinstra~\cite{BT})}
	Let $G$ be a connected graph and let $k\geq 2$ be an integer. If
	$\sigma_2(G)\geq |G|-k+1$, then $G$ has a spanning $k-$ended tree.
\end{theorem}
\noindent They also gave an example to show that the condition "$\sigma_2(G)\geq |G|-k+1$" is sharp. For the aim of this paper, we recall such example: For two positive integers $k, m (k\geq 2),$ let $K_{m,m+k}$ be a complete bipartite graph with bipartition $V(K_{m,m+k})=A\cup B, |A|=m, |B|=m+k.$ Set $G:=K_{m,m+k}.$ Then $\sigma_2(G)=|G|-k$ but $G$ has no spanning $k-$ended tree.

A natural question is whether we can find all graphs $G$ so that $\sigma_2(G)\geq |G|-k$ but $G$ has no spanning $k-$ended tree. In this paper, we will give an answer for this question. As its applications, we give an improvement of Theorem \ref{t2}.
\section{Main results}
To state the main results, we first define $H$ to be the graph obtained from $K_{m,m+k}$ by adding (or no adding) some edges with two end vertices in $A$ (see Figure 1).
 
  \begin{figure}[h]
 	\centering
 	\includegraphics[width=0.5\linewidth]{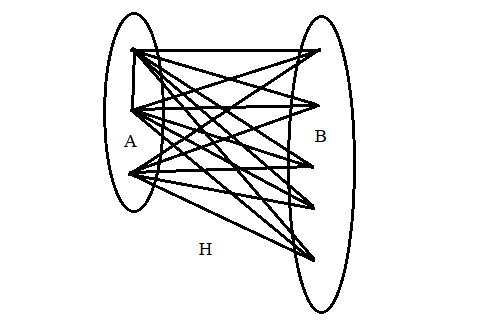}
 	\caption[Graph H]{Graph H with m=3, k=2}
 	\label{Pic1}
 \end{figure}
 
  We have the following lemma.
\begin{lemma} \label{lemma}
	$\sigma_2(H)= |H|-k$ and $H$ has no spanning $k-$ ended tree.
\end{lemma}
\begin{proof}
	By the definition of the graph $H,$ we have the fact that for each two vertices $x\in A, y \in B$ then $\deg_H(x)\geq m+k$ and $ \deg_H(y)=m.$ Moreover, for each two distinct vertices $x, y \in B$ then $xy\not\in E(H).$ Hence we can compute $\sigma_2(H)=2m= |H|-k.$
	 
	 Now, assume that $T$ is a spanning tree of $H.$ Then, we have $E(T)= V(T)- 1= 2m+k-1.$
On the other hand, we have 
\begin{align*}
	E(T)& \geq |L(T)\cap B| + 2|B-(L(T)\cap B)|\\
	&=2|B|-|L(T)\cap B|=2(m+k)-|L(T)\cap B|.	
	\end{align*}
Hence, we obtain 
$$
	2m+k-1 \geq 2(m+k)-|L(T)\cap B|\Rightarrow |L(T)\cap B| \geq k+1 \Rightarrow |L(T)| \geq k+1.	
$$
Therefore, $T$ is not a spanning $k-$ended tree.\\
 The lemma is proved.
\end{proof}
The main purpose of this paper is to prove the following theorem.
\begin{theorem}\label{thm-main}
Let $G$ be a connected graph and let $k\geq 2$ be an integer. If $\sigma_2(G) \geq |G|-k$ then $G$ has a spanning $k-$ended tree except for the case $G$ is isomorphic to a graph $H.$	
		\end{theorem}
\begin{remark}
	We can see that Theorem \ref{t2} is a corollary of Theorem \ref{thm-main} and Lemma \ref{lemma}.
\end{remark}
Moreover, it is easy to see that if a tree has at most $k$ leaves ($k\geq 2$),
then it has at most $k-2$ branch vertices. Therefore, we immediately
obtain the following corollary from Theorem~\ref{thm-main} and the fact that $H$ has a spanning tree with one branch vertex.
\begin{corollary}\label{coro1}
Let $G$ be a connected graph and let $k\geq 3$ be an integer. If $\sigma_2(G) \geq |G|-k$ then $G$ has a spanning with at most $k-2$ branch vertices.
\end{corollary}	
\section{Proofs of Theorem \ref{thm-main}}
In this section, we always denote by $P[u,v]$ a path connecting two vertices $u$ and $v$, which are the end vertices of $P$. Let $P$ be a longest path in $G$ and let $u,v$ be the end-vertices of $P$. We assign an orientation in $P[u, v]$ from $u$ to $v$, and for a vertex $x$ of $P$, we denote its successor and predecessor, if any, by $x^+$ and $x^-$, respectively. For each vertex $z \in G,$ we denote $d_G(z,P)=\min\{ d_G(z,x): x \in P\}.$ A tree is called a caterpillar if all its leaves are adjacent to the
same its path, and the path is called a spine of the caterpillar. 

Suppose that $G$ satisfies $\sigma_2(G) \geq |G|-k$ but $G$ contains no
spanning $k-$ended tree. We will prove that $G$ is isomorphic to a graph $H.$

By the assumption, $G$ has no a Hamiltonian path. Then, $P$ is not a Hamiltonian path of $G.$ Hence, the following claim holds immediately by the fact that $P$ is a longest path of $G$.
\begin{claim}\label{claim1}
	\begin{enumerate}
		\item[{\rm (i)}] $N_G(u)\cup N_G(v) \subseteq V(P)$.
		\item[{\rm (ii)}] $G$ has no cycle $C$ with $V(C) =V(P)$.
		\item[{\rm (iii)}] $N_G(u)^- \cap N_G(v)=\emptyset$ and $\{v\} \cup N_G(u)^- \cup N_G(v) \subseteq V(P)$.
	\end{enumerate}	
\end{claim}
By Claim \ref{claim1} (ii) we obtain $uv \not\in E(G).$ By Claim \ref{claim1} (iii), we have 
$$
\deg_G(u) +\deg_G(v) = |N_G(u)^-| +|N_G(v)| \leq |P|-1. 
$$
Then 
$$
|G|-k\leq \sigma_2{(G)} \leq \deg_G(u) +\deg_G(v) \leq |P|-1 \Rightarrow |G|-  |P|\leq k-1. 
$$
On the other hand, $G$ has no spanning $k-$ended tree. Hence, we obtain that 
\begin{equation} \label{eq1}
|G|=|P|+k-1,
\end{equation}
and for every vertex $w\in V(G)-V(P)$ then $d_G(w, P)=1$ and $w$ must be a leaf of $T.$\\
Moreover, $G$ also has a spanning caterpillar $T$ with spine $P$ and 
\begin{equation} \label{eq2}
\{v\} \cup N_G(u)^- \cup N_G(v) = V(P).
\end{equation}

Set $V(G)-V(P)=\{x_1,...,x_{k-1}\}$ and $x_k=v.$ Hence, $\{u,x_1,...,x_{k}\}$ is the
set of leaves of $T.$
\begin{claim}\label{claim2}
	$\{u,x_1,...,x_{k}\}$ is an independent set in $G$.
\end{claim}
\begin{proof}
	By Claim \ref{claim1} (i), we see that $ux_i\not\in E(G)$ and $x_kx_i \not\in E(G)$ for all $1\leq i \leq k-1.$ Now assume that $x_ix_j \in E(G)$ for some $1\leq i < j \leq k-1.$ We call $z\in P$ the vertex adjacent to $x_j$ in $T.$ Then the tree $T':=T+x_ix_j - zx_j$ is a spanning $k-$ended tree of $G.$ This gives a contradiction with the assumption. 
	\end{proof}
\begin{claim}\label{claim3}
For each $i\in \{1,...,k-1\}$, then $N_G(x_i)=N_G(v)(=N_G(x_k)).$
\end{claim}
\begin{proof}
	By Claim \ref{claim2}, we have $N_G(x_i) \subseteq V(P)-\{u, v\}.$ \\
	If there exists some vertex $y \in N_G(x_i)\cap N_G(u)^{-}$ then $P':=P+x_iy+uy^{+}-yy^{+}$ is a path with $|P'|=|P|+1.$ This contradicts the maximality of $P.$ Hence $N_G(x_i)\cap N_G(u)^{-} =\emptyset.$ Therefore, we obtain 
	$$
	\deg_G(u) +\deg_G(x_i)= |N_G(u)| +|N_G(x_i)| = |N_G(u)^-| +|N_G(x_i)| \leq |P|-1. 
	$$
	Combining with Claim \ref{claim2}, we have
	$$
	|G|-k\leq \sigma_2{(G)}\leq \deg_G(u) +\deg_G(x_i) \leq |P|-1 \Rightarrow |G| \leq |P|+k-1. 
	$$
	On the other hand, by (\ref{eq1}) we have $|G|=|P|+k-1.$ Then the equalities happen. Hence $|N_G(u)^-| +|N_G(x_i)| = |P|-1. $ Therefore we conclude that $N_G(u)^-\cup N_G(x_i) = V(P)-\{v\}.$ By combining with (\ref{eq2}), Claim \ref{claim1} (iii) and $N_G(x_i)\cap N_G(u)^{-} =\emptyset,$ we obtain $N_G(x_i)=N_G(v).$ This completes the proof of Claim \ref{claim3}.
\end{proof}
\begin{claim}\label{claim4}
	For every two distinct vertices $y, z \in N_G(v)$, then $yz \not\in E(P).$
\end{claim}
\begin{proof}
	Suppose to the contrary that there exist two distinct vertices $y, z \in N_G(v)$ such that $yz \in E(P).$ By Claim \ref{claim3} we have $yx_1 \in E(G)$ and $zx_1\in E(G).$ We consider the path $P':=P+yx_1+zx_1-yz.$ Then $|P'| > |P|,$ this is a contradiction with the maximality of $P.$ Therefore Claim \ref{claim4} is proved.
\end{proof}
Set $A=N_G(v)$ and $m=|N_G(v)|=|A|.$\\
By Claim \ref{claim2} and Claim \ref{claim3} we have
\begin{equation} \label{eq3}
\deg_G(x_1) +\deg_G(x_k)= m+m = 2m \geq |G|-k = |P|-1 \Rightarrow 2m\geq|P|-1.
\end{equation}
On the other hand, by Claim \ref{claim4} and $N_G(v) \subseteq V(P)-\{u, v\}$ we obtain $|N_G(v)| \leq \dfrac{|P|-1}{2}.$ Hence,
 \begin{equation} \label{eq4}
 2m \leq |P|-1.
 \end{equation}
 Using (\ref{eq3}) and (\ref{eq4}) we conclude $2m = |P|-1.$ We thus obtain
  \begin{equation} \label{eq5}
 |G|=2m+k.
 \end{equation}
 Moreover, we also obtain $|N_G(u)^{-}|=|P|-1-m=m.$ 
\begin{claim}\label{claim5}
	For every two distinct vertices $x, y \in N_G(u)^{-}$, then $xy \not\in E(G).$ In particular, $N_G(u)^{-}$ is an independent set in $G.$
\end{claim}
\begin{proof}
Set $V(P[u,v])=\{u=z_0,z_1,..., z_{2m-1}, z_{2m}=v\}$ such that $z_i=z_{i+1}^{-}$ for every $0\leq i \leq 2m-1.$ By combining with (\ref{eq2}), Claim \ref{claim1} (iii) and Claim \ref{claim4}, we have $N_G(v)=\{z_{2j+1}\ | \ 0\leq j \leq m-1\}$ and $N_G(u)^{-}=\{z_{2j}\ |\ 0\leq j \leq m-1\}.$ In particular, we obtain
\begin{equation} \label{eq6}
N_G(u)=\{z_{2j+1}\ |\ 0\leq j \leq m-1\}=N_G(v).
\end{equation} 
Now, suppose the assertion of the claim is false. By (\ref{eq6}), there exist $1\leq i < j \leq m-1$ such that $z_{2i}z_{2j}  \in E(G).$ Hence we consider the cycle $C:=P+uz_{2i+1}+z_{2i}z_{2j}+z_{2j-1}v-z_{2i}z_{2i+1}-x_{2j}x_{2j-1}.$ Then the cycle $C$ has $|C| = |P|,$ this is a contradiction with Claim \ref{claim1}(ii). Claim \ref{claim5} is proved.
\end{proof}
\begin{claim}\label{claim6}
	For each vertex $y \in N_G(u)^{-}$, then $N_G(y)=N_G(v)(=A).$
\end{claim}
\begin{proof}
	By (\ref{eq6}), it is true for the case $y=u.$ Now for each vertex $y \in N_G(u)^{-},(y\not=u).$ By Claim \ref{claim3} and (\ref{eq6}), we obtain $N_G(x_i) =N_G(v)$ for all $0\leq i \leq k.$ By combining with Claim \ref{claim5} we have $N_G(y) \subseteq N_G(v).$ Hence
	\begin{equation*}
	 2m=|G|-k \leq \sigma_2{(G)} \leq \deg_G(u) +\deg_G(y)\leq 2|N_G(v)|= 2m.
	\end{equation*}
	Then the equalities happen, in particular, we have $|N_G(y)|=m=|N_G(v)|.$ So we obtain $N_G(y)=N_G(v).$ Claim \ref{claim6} is proved.
\end{proof}
Now we set $B=\{x_i\}_{1\leq i \leq k}\cup (N_G(u)^{-}).$ Then $|B|=m+k.$ By Claim \ref{claim2}, Claim \ref{claim3} and Claim \ref{claim6} we obtain $B$ is an independent set in $G$ and $N_G(y)=A$ for every vertex $y\in B.$ These imply that $G$ is isomorphic to a graph $H$ with bipartition $V(H)=A\cup B.$

 Therefore we conclude that if $\sigma_2(G) \geq |G|-k$ and $G$ contains no spanning $k-$ended tree then $G$ is isomorphic to a graph $H.$ This completes the proof of Theorem \ref{thm-main}.

\end{document}